\documentclass[a4paper,12pt]{amsart}

\usepackage{graphicx}
\usepackage{amsmath,amssymb,amsthm,amsfonts}
\usepackage{amssymb}

\newtheorem{thm}{Theorem}[section]

\newtheorem{lem}[thm]{Lemma}
\newtheorem{prop}[thm]{Proposition}
\newtheorem{rem}[thm]{Remark}

\newcommand{\bremark}{\begin{rem} \textup}
\newcommand{\eremark}{\end{rem} }

\newcommand{\cuad}{{\sqcap\kern-.68em\sqcup}}

\renewcommand{\rho}{\varrho}
\renewcommand{\theta}{\vartheta}

\begin{document}


\parindent 0pc
\parskip 6pt
\overfullrule=0pt

\title[The moving plane method]{The moving plane method for singular semilinear elliptic problems}

\author{Annamaria Canino}
\address{Annamaria Canino,
Universit\`a della Calabria\\
Dipartimento di Matematica e Informatica\\
Pietro Bucci 31B, I-87036 Arcavacata di Rende, Cosenza, Italy}
\email{canino@mat.unical.it}

\author{Luigi Montoro}
\address{Luigi Montoro,
Universit\`a della Calabria\\
Dipartimento di Matematica e Informatica\\
Pietro Bucci 31B, I-87036 Arcavacata di Rende, Cosenza, Italy}
\email{montoro@mat.unical.it}

\author{Berardino Sciunzi}
\address{Berardino Sciunzi,
Universit\`a della Calabria\\
Dipartimento di Matematica e Informatica\\
Pietro Bucci 31B, I-87036 Arcavacata di Rende, Cosenza, Italy}
\email {sciunzi@mat.unical.it}

\date{\today}




\begin{abstract}
We consider positive solutions  to  semilinear elliptic problems with singular nonlinearities, under zero Dirichlet boundary condition.
We exploit a refined version of the moving plane method to prove symmetry and monotonicity properties of the solutions, under general assumptions on the nonlinearity.
\end{abstract}

\maketitle

\date{\today}



\maketitle

\section{introduction}

In this paper we study symmetry and monotonicity properties of positive  solutions to the problem
\begin{equation}
	\label{problem}
\begin{cases}
-\Delta\,u\,=\,\frac{1}{u^\gamma}\,+\,f(x,u) & \text{in $\Omega$,}  \\
u> 0 & \text{in $\Omega$,}  \\
u=0 & \text{on $\partial\Omega$.}
\end{cases}
\end{equation}
where $\gamma>0$, $\Omega $ is a bounded smooth domain of $\mathbb{R}^n$ and
$u\in\,C(\overline{\Omega})\cap C^2(\Omega)$. \\

Starting from the pioneering work \cite{crandall} singular semilinear elliptic equations have been intensely studied, see e.g. \cite{boccardo2,boccardo,C23,CanDeg,gras1,GOW,saccon,lair,lazer,kaw,marco,stuart}. Furthermore,
by a simple change of variables, it also follows that the problem  is  related to equations involving a first order term of the type $\frac{|\nabla u|^2}{u}$. We refer the readers to \cite{arcoya,brandolini,giachetti} for related results in this setting.\\

The main difficulties that we have to face are given by the fact that solutions in general are not in $H^1_0(\Omega)$ and the nonlinearity
$\frac{1}{s^\gamma}\,+\,f(x,s)$ is not  Lipschitz continuous at zero.
Note that solutions  are not in $H^1_0(\Omega)$ already
in the case $f\equiv 0$, see  \cite{lazer}. Therefore, in particular,
problem \eqref{problem} has to be understood in the weak distributional meaning with test functions with compact support in $\Omega$, that is
\begin{equation}\label{eq:fgakghfhksajdgfjashagc}
\int_\Omega\,(\nabla u,  \nabla \varphi)\,dx  = \int_\Omega \frac{\varphi}{{u}^\gamma}\,dx + \int_\Omega f(x,u){\varphi}\,dx \qquad\forall \varphi\in C^1_c(\Omega).
\end{equation}

The proof of our symmetry result will be based on the moving plane technique, see \cite{GNN,serrin}, as developed and improved in \cite{BN}.
The crucial point here is the lack of regularity of the solutions near the boundary, that  is an obstruction to the use of  the test functions technique exploited in \cite{BN,GNN,serrin}.

As we will see,
a special role in this issue is plaid by $u_0$, the solution  to the pure singular problem:
$u_0\in\,C(\overline{\Omega})\cap C^2(\Omega)$ and
\begin{equation}
	\label{problem0}
\begin{cases}
-\Delta\,u_0\,=\,\frac{1}{{u_0}^\gamma}\, & \text{in $\Omega$,}  \\
u_0> 0 & \text{in $\Omega$,}  \\
u_0=0 & \text{on $\partial\Omega$.}
\end{cases}
\end{equation}

\noindent The solution $u_0$  is  unique (see  \cite{CanDeg,ccmcan,nodt,oliva}) and the existence has been proved in \cite{boccardo,CanDeg}.
By the variational characterization provided in \cite{CanDeg}, it follows that any solution $u$ to problem \eqref{problem} enjoy the decomposition
\[
u=u_0+w\qquad\text{for some}\quad w\in H^1_0(\Omega)\,.
\]
Such a decomposition has been exploited in \cite{gras1} (see also the applications in \cite{C23,Cancan,gras}) in order to prove symmetry and monotonicity properties of the solution, via a moving plane type technique applied to $w$, the $H^1_0(\Omega)$ part of the solution. Since $w$ is not a solution to the problem, such approach required an extra condition on the nonlinearity $f(x,u)$ that, in  \cite{gras1},  is assumed to be monotone increasing in the $u$ variable.\\

The aim of this paper is to remove such a restriction on the nonlinearity and prove symmetry and monotonicity properties of the solution under general assumptions, namely in the case of locally Lipschitz continuous nonlinearities that, more precisely,  fulfill

\begin{itemize}
\item[($hp$)]
$f(x,t)$ is a Carath\'{e}odory function which is uniformly locally Lipschitz continuous with respect to the second variable. Namely, for any $M>0$ given, it follows
\[
|f(x,t_1)-f(x,t_2)|\leq L_f(M)|t_1-t_2|,\quad x\in\Omega,\quad t_1,t_2\in[0\,,\,M].
\]
\end{itemize}

 Our main result is the following
\begin{thm}\label{t3}
Let $u\in\,C(\overline{\Omega})\cap C^2(\Omega)$ be  a solution to \eqref{problem}.  Assume that the domain $\Omega$ is   convex w.r.t. the $\nu$-direction $(\nu \in S^{N-1})$ and symmetric w.r.t. $T_0^\nu$, where
$$ T_0^\nu=\{x\in \mathbb{R}^N : x\cdot \nu=0\}.$$
With the notation $
x_\lambda^\nu=R_\lambda^\nu(x)=x+2(\lambda -x\cdot\nu)\nu$,
assume that  $f$ satisfies $(hp)$, $f(\cdot, t)$ is non decreasing in the $x\cdot \nu$-direction in the set
$\Omega_0^\nu\,:=\,\Omega\cap\{x\cdot \nu<0\},$
for all $t\in [0,\infty)$ and
$$f(x,t)= f(x_0^\nu,t) \quad  \text{if} \,\,x\in\Omega_0 \,\, \text{and} \,\,t\in [0,\infty).$$

Then $u$ is symmetric w.r.t. $T_0^\nu$ and non-decreasing w.r.t. the $\nu$-direction in~$\Omega_0^\nu$.
In particular, if  $\Omega$ is a ball centered at the origin of radius $R>0$, then $u$ is radially symmetric with $\frac{\partial u}{\partial r}(r)<0$ for
$ 0<r<R$.
\end{thm}

The key point in the proof of Theorem \ref{t3} is the study of the problem near the boundary. We combine a fine analysis of the behaviour of the solution near the boundary based on comparison arguments that go back to \cite{CanDeg}, with an improved test functions technique. Let us finally point out that, the monotonicity assumption on $f$, with respect to the first variable, is necessary for the applicability of the moving plane method. This is  well known already in the case of non singular nonlinearities.

\section{The symmetry result}
To state the next results we need some notations. Let $\nu$ be a direction in $\mathbb{R}^N$ with $|\nu|=1$. Given a real number $\lambda$ we set
\begin{equation}\nonumber
T_\lambda^\nu=\{x\in \mathbb{R}^N:x\cdot\nu=\lambda\},
\end{equation}
\begin{equation}\nonumber
\Omega_\lambda^\nu=\{x\in \Omega:x\cdot \nu <\lambda\}
\end{equation}
and
\begin{equation}\nonumber
x_\lambda^\nu=R_\lambda^\nu(x)=x+2(\lambda -x\cdot\nu)\nu,
\end{equation}
that is the reflection trough the hyperplane $T_\lambda^\nu$. Moreover we set
\begin{equation}\nonumber
(\Omega_\lambda^\nu)'=R_\lambda^\nu(\Omega_\lambda^\nu).
\end{equation}
Observe that $(\Omega_\lambda^\nu)'$ may be not contained in $\Omega$. Also we take
\begin{equation}\nonumber
a(\nu)=\inf _{x\in\Omega}x\cdot \nu.
\end{equation}
When $\lambda >a(\nu)$, since $\Omega_\lambda^\nu$ is nonempty, we set
\begin{equation}\nonumber
\Lambda_1(\nu)=\{\lambda : (\Omega_t^\nu)'\subset \Omega\,\, \text{for any}\,\,a(\nu)< t\leq\lambda\},\end{equation}
and
\begin{equation}\nonumber
\lambda_1(\nu)=\sup \Lambda_1(\nu).
\end{equation}
Moreover we set
\begin{equation}\nonumber
u_\lambda^\nu(x)=u(x_\lambda^\nu)\,,
\end{equation}
 for any $a(\nu)<\lambda\leq \lambda_1(\nu)$.
Moreover let us define
\begin{equation}\nonumber
d(x)=\text{dist}(x,\partial \Omega), \qquad \forall x\in \mathcal I_\delta(\partial \Omega),
\end{equation}
whit $I_\delta(\partial \Omega)$ a neighborhood of radius $\delta>0$ of $\partial \Omega$, with the \emph{unique nearest point property}, see \cite{carlo} and the references therein. We start proving the following
\begin{lem}\label{lem:viva}
Let $u$ be a solution to \eqref{problem}. Then \begin{equation}\nonumber
u(x)\leq C\,d(x)^{\frac{1}{\gamma+1}} \qquad \text{in } \,\,\mathcal I_\delta(\partial \Omega),
\end{equation}
for some positive constant
 $C= C(f,\gamma,\delta,\Omega,\|u\|_{L^{\infty}(\Omega)})$.
\end{lem}
\begin{proof} Since $u\in C(\overline{\Omega})$ and $f$ satisfies $(hp)$, using \eqref{problem}, we obtain in the weak distributional meaning
\begin{equation*}
-\Delta u\leq \frac{C}{u^{\gamma}} \qquad \text{in } \,\, \Omega,
\end{equation*}
for some positive constant $C=C(f,\gamma,\Omega,\|u\|_{L^{\infty}(\Omega)})$. By \cite[Theorem 2.2, Lemma 2.8]{CanDeg} it follows that
\begin{equation*}
u(x)\leq C u_1(x)^{\frac{1}{\gamma+1}} \qquad \text{in } \,\, \Omega,
\end{equation*}
where $u_1$ is the solutions to $-\Delta u_1=1$ in $\Omega$ with zero Dirichlet boundary condition. Since $u_1\in C^1(\overline \Omega)$, the result follows by the mean value theorem.
\end{proof}
In the following we will  denote by $\chi(A)$  the characteristic function of a set $A$ and, with no loss of generality, we will assume that $\nu =e_1$. We have
\begin{prop}\label{prop:alpha>gamma}
For any $\lambda <0$ we have that
\begin{equation*}
[(u-u_\lambda)^+]^s\cdot \chi(\Omega_\lambda)\in H^1_0(\Omega_\lambda),
\end{equation*}
where $\Omega_\lambda :=\{x\in \Omega\,:\, x_1\leq\lambda\}$, provided that
$$s\geq \max\{\frac{\gamma+1}{2}\,,\,1\}
.$$
\end{prop}
\begin{proof}
Let  $g_{\varepsilon}(t):\mathbb{R}^+\rightarrow \mathbb{R}^+$ be locally Lipschitz continuous and  such that
\begin{equation}\nonumber
\begin{cases}
g_{\varepsilon}(t)=0 & \text{in $[0,\varepsilon]$,}  \\
g_{\varepsilon}(t)=1 & \text{in $[2\varepsilon, +\infty)$,} \\
g'_{\varepsilon}(t)\leq \frac{C}{\varepsilon} & \text{in $(\varepsilon, 2\varepsilon)$.}
\end{cases}
\end{equation}
We set
\begin{equation}\nonumber
\varphi_\varepsilon (x):=
\begin{cases}
g_\varepsilon(d(x))& \text{in $\mathcal I_\delta(\partial \Omega)$,}  \\
1& \text{in $\Omega \setminus \mathcal I_\delta(\partial \Omega)$,}  \\
\end{cases}
\end{equation}
where it is convenient to choose $\varepsilon >0$ such that  $2\varepsilon <\delta$. We note that
\begin{equation}\label{eq:suppppnablafi}
{supp}\, |\nabla \varphi_\varepsilon|\subseteq \Big\{x\in \mathcal I_\delta \,:\, \varepsilon<d(x)<2 \varepsilon \Big\}
\end{equation}
and almost everywhere
$$|\nabla \varphi_{\varepsilon}(x)|\leq |g'_\varepsilon(d(x))||\nabla d(x)|\leq \frac{C}{d(x)}.$$
Let us consider
\begin{equation*}
\Psi_\varepsilon=[(u-u_\lambda)^+]^{\alpha}\,\varphi_\varepsilon^2\,\chi(\Omega_\lambda),
\end{equation*}
with $\alpha>1$ to be chosen later. By \eqref{eq:fgakghfhksajdgfjashagc} we deduce that
\begin{equation}\label{eq:fgakghfhksajdgfjashagcc}
\int_{R_\lambda(\Omega)}\,(\nabla u_\lambda,  \nabla \varphi)\,dx  = \int_{R_\lambda(\Omega)} \frac{\varphi}{({u_\lambda})^\gamma}\,dx + \int_{R_\lambda(\Omega)}f(x_\lambda,u_\lambda){\varphi}\,dx \qquad\forall \varphi\in C^1_c({R_\lambda(\Omega)}),
\end{equation}
as well. By standard density arguments it follows that we can plug $\Psi_\varepsilon$ as test function  in \eqref{eq:fgakghfhksajdgfjashagc} and  in \eqref{eq:fgakghfhksajdgfjashagcc} and then,  subtracting, we obtain
\begin{eqnarray}\nonumber
&&\alpha \int_{\Omega_\lambda}\,|\nabla (u-u_\lambda)^+|^2[(u-u_\lambda)^+]^{\alpha-1}\varphi_\varepsilon^2\,dx  \\\nonumber
&&\leq 2  \int_{\Omega_\lambda}\,|\nabla (u-u_\lambda)^+||\nabla \varphi_\varepsilon|\varphi_\varepsilon [(u-u_\lambda)^+]^{\alpha}\,dx
\\\nonumber &&+  \int_{\Omega_\lambda}\Big (u^{-\gamma}-u_\lambda^{-\gamma}\Big)[(u-u_\lambda)^+]^{\alpha}\,\varphi_\varepsilon^2\,dx
\\\nonumber &&+  \int_{\Omega_\lambda}\Big (f(x,u)-f(x_\lambda,u_\lambda)\Big)[(u-u_\lambda)^+]^{\alpha}\,\varphi_\varepsilon^2\,dx
\\\nonumber &&\leq 2  \int_{\Omega_\lambda}\,|\nabla (u-u_\lambda)^+||\nabla \varphi_\varepsilon|\varphi_\varepsilon [(u-u_\lambda)^+]^{\alpha}\,dx \\\nonumber &&
\\\nonumber
&&+ \int_{\Omega_\lambda}\Big (f(x,u)-f(x,u_\lambda)\Big)[(u-u_\lambda)^+]^{\alpha}\,\varphi_\varepsilon^2\,dx,
\end{eqnarray}
where we used that  $f(\cdot, t)$ is non decreasing in the $x_1$-direction in $\Omega_0$ and that
$u^{-\gamma}-u_\lambda^{-\gamma}\leq 0$ in the support of $(u-u_\lambda)^+$. Moreover by the assumption $(hp)$
\begin{eqnarray}\label{eq:purpleaksjbskasj}
&&\alpha \int_{\Omega_\lambda}\,|\nabla (u-u_\lambda)^+|^2[(u-u_\lambda)^+]^{\alpha-1}\varphi_\varepsilon^2\,dx  \\\nonumber
&&\leq 2\int_{\Omega_\lambda}\,|\nabla (u-u_\lambda)^+||\nabla \varphi_\varepsilon|\varphi_\varepsilon[(u-u_\lambda)^+]^{\alpha}\,dx \\\nonumber &&
\\\nonumber
&&+ C(f, \|u\|_{L^{\infty}(\Omega)})\int_{\Omega_\lambda}[(u-u_\lambda)^+]^{\alpha+1}\,\varphi_\varepsilon^2\,dx
\\\nonumber
&&\leq 2\int_{\Omega_\lambda}\,|\nabla (u-u_\lambda)^+||\nabla \varphi_\varepsilon|\varphi_\varepsilon[(u-u_\lambda)^+]^{\alpha}\,dx+C(f,\alpha,  \|u\|_{L^{\infty}(\Omega)}).
\end{eqnarray}
By weighted Young inequality \eqref{eq:purpleaksjbskasj} becomes
\begin{eqnarray}\label{eq:purpleaksjbskasjj}
&&\frac{\alpha}{2}\int_{\Omega_\lambda}\,|\nabla (u-u_\lambda)^+|^2[(u-u_\lambda)^+]^{\alpha-1}\varphi_\varepsilon^2\,dx  \\\nonumber
&&\leq C(\alpha) \int_{\Omega_\lambda}\,|\nabla \varphi_\varepsilon|^2[(u-u_\lambda)^+]^{\alpha+1}\varphi_\varepsilon^2\,dx + C(f,\alpha,  \|u\|_{L^{\infty}(\Omega)}).
\end{eqnarray}
Using Lemma \ref{lem:viva} and  \eqref{eq:suppppnablafi}, we obtain
\begin{eqnarray}\label{eq:tatatangejedadklchs}
&&\int_{\Omega_\lambda}\,|\nabla \varphi_\varepsilon|^2[(u-u_\lambda)^+]^{\alpha+1}\varphi_\varepsilon^2\,dx= \int_{\Omega_\lambda\cap {supp}\,\nabla \varphi_\varepsilon}\,|\nabla \varphi_\varepsilon|^2[(u-u_\lambda)^+]^{\alpha+1}\varphi_\varepsilon^2\,dx\\\nonumber
&&\leq C\int_{\Omega_\lambda\cap {supp}\,\nabla \varphi_\varepsilon}\,\Big(d(x)\Big)^{-2}\Big(d(x)\Big)^{\frac{\alpha+1}{\gamma +1}}\,dx\leq C\varepsilon^{\frac{\alpha+1}{\gamma +1}-2}\mathcal L(\Omega_\lambda\cap {supp}\,\nabla \varphi_\varepsilon)
\end{eqnarray}
where by $\mathcal L (A)$ we denote the Lebesgue measure of a measurable set $A$. Moreover, since  $\Omega_\lambda\cap {supp}\,\nabla \varphi_\varepsilon\subset I_\delta(\partial \Omega)$, then $\mathcal L(\Omega_\lambda\cap {supp}\,\nabla \varphi_\varepsilon)\leq C\varepsilon$, for some positive constant $C=C(\Omega)$. Finally from \eqref{eq:purpleaksjbskasjj} and \eqref{eq:tatatangejedadklchs} we get
\begin{equation*}
\int_{\Omega_\lambda}\,|\nabla (u-u_\lambda)^+|^2[(u-u_\lambda)^+]^{\alpha-1}\varphi_\varepsilon^2\,dx\leq C  \end{equation*}
with $C= C(f,\alpha,\gamma,\delta,\Omega,\|u\|_{L^{\infty}(\Omega)})$, if $\alpha \geq \gamma$. By Fatou's Lemma we obtain
$$[(u-u_\lambda)^+]^{\frac{\alpha +1}{2}}\in H^1_0(\Omega_\lambda), \qquad \text{if}\,\, \alpha \geq \gamma.$$
\end{proof}
\noindent \emph{Proof of Theorem \ref{t3}.} Using the same notations of the proof of Proposition~\ref{prop:alpha>gamma} and arguing  as above, we consider
\begin{equation*}
\Psi_\varepsilon=[(u-u_\lambda)^+]^{\beta}\,\varphi_\varepsilon^2\,\chi(\Omega_\lambda),
\end{equation*}
with
\begin{equation}\label{eq:beta}
\beta>\max \Big\{1,\gamma, (\gamma+1)/2\Big\}.
\end{equation}
By density arguments we plug  $\Psi_\varepsilon$ as test function  in \eqref{eq:fgakghfhksajdgfjashagc} and in  \eqref{eq:fgakghfhksajdgfjashagcc} and then,  subtracting, we get that
\begin{eqnarray}\label{eq:equazioneinbeta}
&&\beta \int_{\Omega_\lambda}\,|\nabla (u-u_\lambda)^+|^2[(u-u_\lambda)^+]^{\beta-1}\varphi_\varepsilon^2\,dx \\\nonumber &&\leq 2  \int_{\Omega_\lambda}\,|\nabla (u-u_\lambda)^+||\nabla \varphi_\varepsilon|\varphi_\varepsilon [(u-u_\lambda)^+]^{\beta}\,dx \\\nonumber &&
\\\nonumber &&+  \int_{\Omega_\lambda}\Big (u^{-\gamma}-u_\lambda^{-\gamma}\Big)[(u-u_\lambda)^+]^{\alpha}\,\varphi_\varepsilon^2\,dx
\\\nonumber&&+ \int_{\Omega_\lambda}\Big (f(x,u)-f(x,u_\lambda)\Big)[(u-u_\lambda)^+]^{\beta}\,\varphi_\varepsilon^2\,dx, \\\nonumber
 &&\leq 2  \int_{\Omega_\lambda}\,|\nabla (u-u_\lambda)^+||\nabla \varphi_\varepsilon|\varphi_\varepsilon [(u-u_\lambda)^+]^{\beta}\,dx \\\nonumber &&
\\\nonumber
&&+C(f, \|u\|_{L^{\infty}(\Omega)}) \int_{\Omega_\lambda}[(u-u_\lambda)^+]^{\beta+1}\,\varphi_\varepsilon^2\,dx.
\end{eqnarray}
We  estimate the first term on the right-hand side of the last line of \eqref{eq:equazioneinbeta}  as follows. Using H\"older inequality and then Proposition \eqref{prop:alpha>gamma} we obtain
\begin{eqnarray}\nonumber
 &&\int_{\Omega_\lambda}\,|\nabla (u-u_\lambda)^+||\nabla \varphi_\varepsilon|\varphi_\varepsilon [(u-u_\lambda)^+]^{\beta}\,dx\\\nonumber
 &&\leq C(\gamma)\left(\int_{\Omega_\lambda}\,\Big|\nabla [(u-u_\lambda)^+]^{\frac{\gamma+1}{2}}\Big|^2\, dx\right)^\frac{1}{2}\left(\int_{\Omega_\lambda}\,|\nabla \varphi_\varepsilon|^2[(u-u_\lambda)^+]^{2\beta-(\gamma -1)}\,dx\right)^\frac{1}{2}\\\nonumber
&&\leq  C\left(\int_{\Omega_\lambda}\,|\nabla \varphi_\varepsilon|^2[(u-u_\lambda)^+]^{2\beta-(\gamma -1)}\,dx\right)^\frac{1}{2},
\end{eqnarray}
with $C= C(f,\gamma,\delta,\Omega,\|u\|_{L^{\infty}(\Omega)})$. Using \eqref{eq:suppppnablafi} again we infer that
\begin{equation*}
\int_{\Omega_\lambda}\,|\nabla (u-u_\lambda)^+||\nabla \varphi_\varepsilon|\varphi_\varepsilon [(u-u_\lambda)^+]^{\beta}\,dx\leq C\varepsilon^{\frac{2(\beta-\gamma)}{\gamma+1}}
 \end{equation*}
 and then
 \begin{equation}\label{eq:chianoviobobrumore}
\int_{\Omega_\lambda}\,|\nabla (u-u_\lambda)^+||\nabla \varphi_\varepsilon|\varphi_\varepsilon [(u-u_\lambda)^+]^{\beta}\,dx=o(1),\qquad \text{as}\,\, \varepsilon \rightarrow 0,
 \end{equation}
since, by \eqref{eq:beta},  ${2(\beta-\gamma)}/{(\gamma+1)}>0$. Then by \eqref{eq:equazioneinbeta} and \eqref{eq:chianoviobobrumore}, passing to the the limit, we deduce that
\begin{equation}\label{eq:sahhsdjialgidkj}
\int_{\Omega_\lambda}\,|\nabla (u-u_\lambda)^+|^2[(u-u_\lambda)^+]^{\beta-1}\,dx \leq C\int_{\Omega_\lambda}[(u-u_\lambda)^+]^{\beta+1}\,dx,
\end{equation}
with $C=C(\beta, \gamma, f, \|u\|_{L^{\infty}(\Omega)}) $ a positive constant.
As a consequence of Proposition \ref{prop:alpha>gamma}, recalling \eqref{eq:beta}, we can
apply Poincar\'e inequality in the r.h.s of \eqref{eq:sahhsdjialgidkj} to deduce that
\begin{eqnarray}\label{eq:sahhsdjialgidkjq}
&&\int_{\Omega_\lambda}\,|\nabla (u-u_\lambda)^+|^2[(u-u_\lambda)^+]^{\beta-1}\,dx \leq C\int_{\Omega_\lambda}\left([(u-u_\lambda)^+]^{\frac{\beta+1}{2}}\right)^2\,dx\\\nonumber
&&\leq C\cdot C_P(\Omega_\lambda)\int_{\Omega_\lambda}\,|\nabla (u-u_\lambda)^+|^2[(u-u_\lambda)^+]^{\beta-1}\,dx,\end{eqnarray}
were $C_P(\Omega_\lambda)\rightarrow 0$ as $\mathcal L(\Omega_\lambda)\rightarrow 0$. Thus, there exists $\delta=\delta(n,\beta, \gamma, f, \|u\|_{L^{\infty}(\Omega)})$ such that if
\begin{equation}\label{eq:deltedelta}
\mathcal L(\Omega_\lambda)\leq \delta,
\end{equation}
  then $C\cdot C_P(\Omega_{\lambda}) <1$ in \eqref{eq:sahhsdjialgidkjq}. This implies that
\begin{equation}\label{eq:partepositivazero}
(u-u_\lambda)^+=0\qquad \text{in}\,\, \Omega_\lambda,
\end{equation}
namely $u\leq u_\lambda$ in $\Omega_\lambda$. \\

\noindent \emph{Claim}: there exists $\bar \mu>0$ small such that
\begin{equation}\label{eq:partepositivazeroo}
u<u_\lambda \qquad \text{in}\,\, \Omega_\lambda,
\end{equation}
for any  $a(e_1)<\lambda\leq a(e_1)+\bar\mu$.\\

\noindent In fact we can fix $\bar \mu>0$ small so that \eqref{eq:partepositivazero} holds and provides that
\begin{equation}\nonumber
u\leqslant u_\lambda \qquad \text{in}\,\, \Omega_\lambda,
\end{equation}
for any  $a(e_1)<\lambda\leq a(e_1)+\bar\mu$. Therefore we only need to prove the strict inequality.   To prove this
assume by contradiction that, for some $\lambda$, with  $a(e_1)<\lambda\leq a(e_1)+\bar\mu$,  there exists a point $x_0\in\Omega_\lambda$ such that $u(x_0)=u_\lambda(x_0)$. Then let $r=r(x_0)>0$ be such that $B_r(x_0)\subset\subset \Omega_\lambda^\nu$. We have, in the classical sense (since $u\in C^2(\Omega)$),
\begin{equation}\label{eqtocotococoshia}
-\Delta\,(u-u_\lambda)=\left(\frac{1}{u^\gamma}-\frac{1}{u_\lambda^\gamma}\right) +\Big(f(x,u)-f(x_\lambda,u_\lambda)\Big) \qquad \text{in}\,\, B_r(x_0).
\end{equation}
>From \eqref{problem}, we deduce that there exists a positive constant $C=C(r,\lambda)$ such that
$$\min_{x\in B_r(x_0)}\{u(x), u_\lambda(x)\}\geq C>0.$$
Then (using the assumption ($hp$) as well) we can estimate the r.h.s to  \eqref{eqtocotococoshia} as
\begin{equation}\nonumber
\left|\left(\frac{1}{u^\gamma}-\frac{1}{u_\lambda^\gamma}\right) +\Big(f(x,u)-f(x_\lambda,u_\lambda)\Big)\right|\leq C |u-u_\lambda|\qquad \text{in}\,\, B_r(x_0),
\end{equation}
with $C=C(f, r,\lambda,\|u\|_{L^{\infty}(\Omega)})$.  Hence we  find  $\Lambda>0$ such that, from \eqref{eqtocotococoshia}, we obtain
\begin{equation}\nonumber
-\Delta\,(u-u_\lambda)+\Lambda(u-u_\lambda)\geq 0 \qquad \text{in}\,\, B_r(x_0)
\end{equation}
and we are in position to exploit
 the strong maximum principle  \cite{GT} to deduce that $u\equiv u_\lambda$ in $B_r(x_0)$.   By a covering argument it would follow that $u\equiv u_\lambda$ in
$\Omega_\lambda$ providing a contradiction with the Dirichlet condition and thus proving the claim.
\\

\noindent To proceed further we set
\begin{equation}\nonumber
\Lambda_0=\{\lambda > a(e_1): u< u_{t}\,\,\,\text{in}\,\,\,\Omega_t\,\,\,\text{for all
$t\in(a(e_1),\lambda]$}\},
 \end{equation}
 which is not empty thanks to \eqref{eq:partepositivazero}.  Also set
 \begin{equation}\nonumber
\lambda_0=\sup\,\Lambda_0.
\end{equation}
We have to show that actually  $\lambda_0 = \lambda_1(e_1)=0$.
Assume otherwise that $\lambda_0 < 0$ and note that, by continuity, we obtain that $u\leq u_{\lambda_0}$ in~$\Omega_{\lambda_0}$.
 Repeating verbatim the argument used in the proof of the
 previous claim, we deduce that $u< u_{\lambda_0}$ in $\Omega_{\lambda_0}$
unless  $u= u_{\lambda_0}$ in $\Omega_{\lambda_0}$.
But, as above, because of  the zero Dirichlet boundary conditions and since $u>0$ in the interior of the domain,
 the case $u\equiv u_{\lambda_0}$ in $\Omega_{\lambda_0}$ is  not possible if $\lambda_0 <0$. Thus $u< u_{\lambda_0}$ in $\Omega_{\lambda_0}$.\\

Now we fix a compact set $\mathcal{K} \subset\Omega_{\lambda_0}$ so that $\mathcal{L}(\Omega_{\lambda_0}\setminus \mathcal{K})\leq \frac{\delta}{2}$, with $\delta$  given by \eqref{eq:deltedelta}.
By compactness we find $\sigma =\sigma(\mathcal{K})>0$ such that
\[
u_{\lambda_0}-u\geq 2\sigma>0\,\quad \text{in}\quad \mathcal{K}\,.
\]
Take now  $\bar\varepsilon>0$ sufficiently small so that $\lambda_0+\bar\varepsilon<\lambda_1(\nu)$ and
 for any $0<\varepsilon\leq \bar\varepsilon$
\begin{itemize}
\item[$a)$] $u_{\lambda_0+\varepsilon}-u\geq \sigma>0$ in $\mathcal{K}$,
\item[$b)$] $\mathcal{L}(\Omega_{\lambda_0+\varepsilon}\setminus \mathcal{K})\leq \delta\,$.
\end{itemize}
Taking into account $a)$ it is now easy to check  that, for any $0<\varepsilon\leq \bar\varepsilon$, we have that $u\leq u_{\lambda_0+\varepsilon}$ on the boundary of $\Omega_{\lambda_0+\varepsilon}\setminus \mathcal{K}$.  Now we argue as above but considering  the test function
\begin{equation*}
\Psi_\varepsilon=[(u-u_{\lambda_0+\varepsilon})^+]^{\alpha}\,\varphi_\varepsilon^2\,\chi(\Omega_{\lambda_0+\varepsilon}\setminus \mathcal{K}).
\end{equation*}
Following verbatim the arguments  from equation \eqref{eq:beta} to equation \eqref{eq:partepositivazero}, since  $b)$ holds, we obtain
\begin{equation*}
u\leq u_{\lambda_0+\varepsilon}\qquad \text{in}\,\, \Omega_{\lambda_0+\varepsilon}\setminus \mathcal{K}.
\end{equation*}
Thus  $u\leq u_{\lambda_0+\varepsilon}$ in $\Omega_{\lambda_0+\varepsilon}$. We get a contradiction with the definition of $\lambda_0$ and conclude that actually $\lambda_0=\lambda_1(\nu)$.
Then it follows that
$$u(x)\leq u_{0}(x) \,\,  \text{ for } x\in \Omega_0^{e_1}.$$
In the same way, performing the moving plane method in the direction $-e_1$ we obtain
$$u(x)\geq u_0(x) \,\,  \text{ for } x \in \Omega_0^{e_1},$$that is, $u$ is symmetric w.r.t. $T_0^{e_1}$ and non-decreasing w.r.t. the $e_1$-direction in~$\Omega_0^{e_1}$.

\

Finally, if $\Omega$ is a ball of radius $R>0$, repeating the argument for any direction, it follows that $u$ is radially symmetric. The fact that
 $\displaystyle \frac{\partial u}{\partial r}(r)<0$ for $0<r<R$, follows by the Hopf's  Lemma.
\begin{flushright}
$\square$
\end{flushright}

\bigskip

\end{document}